\documentclass{amsart}
\usepackage[english]{babel}
\usepackage[latin1]{inputenc}

\usepackage{newlfont}
\usepackage{amsmath}
\usepackage{graphicx}
\usepackage{amssymb}
\usepackage{amsmath}
\usepackage{latexsym}
\usepackage{amsthm}
\usepackage{stmaryrd}
\usepackage{braket}
\usepackage{enumerate}
\usepackage{booktabs}
\usepackage{mathtools}
\usepackage{array}
\usepackage{tikz}
\usepackage{tikz-cd}
\usepackage{paralist}
\usepackage{color}
\usepackage{cleveref}

\usepackage[a4paper,top=3cm,bottom=3cm,left=3.5cm,right=3.5cm, bindingoffset=5mm]{geometry}

\newtheorem{Theorem}{Theorem}

\newtheorem{Lemma}[Theorem]{Lemma}
\newtheorem{Cor}[Theorem]{Corollary}
\newtheorem{Prop}[Theorem]{Proposition}
\theoremstyle{definition}

\newtheorem{Rem}[Theorem]{Remark}

\DeclareMathOperator{\Spec}{Spec}

\DeclareMathOperator{\Cl}{Cl}
\DeclareMathOperator{\Pic}{Pic}
\DeclareMathOperator{\NS}{NS}

\title{Isolated singularities with large class group}
\author{Alessio Caminata}

\address{{\small Alessio Caminata, Dipartimento di Matematica, Universit\`a di Genova\\ via Dodecaneso 35, 16146, Genova, Italy}}
\email{caminata@dima.unige.it}

\begin{document}


\begin{abstract}
Let $d\geq3$ and $g\geq1$ be integers.
Using a geometric construction involving the symmetric product of a projective curve, we exhibit a $d$-dimensional complete local normal domain over $\mathbb{C}$ with an isolated singularity such that its class group contains a copy of $(\mathbb{R}/\mathbb{Z})^{\oplus 2g}$ as subgroup.
\end{abstract}

\maketitle


We work over the field of complex numbers $\mathbb{C}$. Let $\mathcal{C}$ be a smooth projective irreducible curve of genus $g$, and let $d\geq2$ be an integer. We denote by $\mathcal{C}^{(d)}$ the \emph{$d$-fold symmetric product} of $\mathcal{C}$, that is the quotient of the ordinary $d$-fold product $\mathcal{C}\times\cdots\times\mathcal{C}$ by the natural action of the symmetric group $\mathfrak{S}_d$.
We denote an element of $\mathcal{C}^{(d)}$ by $P_1+\dots+P_d$, where $P_1,\dots,P_d$ are points of $\mathcal{C}$.
It is well known that $\mathcal{C}^{(d)}$ is a smooth projective variety of dimension $d$ (see e.g. \cite[p.~18]{ACGH85}).
Therefore the class group of $\mathcal{C}^{(d)}$ can be identified with the \emph{Picard group} $\Pic(\mathcal{C}^{(d)})$ of isomorphism classes of line bundles on $\mathcal{C}^{(d)}$. Moreover, we denote by $\Pic^0(\mathcal{C}^{(d)})$ the connected component of the identity in $\Pic(\mathcal{C}^{(d)})$ and by $\NS(\mathcal{C}^{(d)})=\Pic(\mathcal{C}^{(d)})/\Pic^0(\mathcal{C}^{(d)})$ the \emph{N\'eron-Severi group} of $\mathcal{C}^{(d)}$, that is the quotient group of numerical equivalence classes of divisors on $\mathcal{C}^{(d)}$.

\par We want to describe the Picard group of $\mathcal{C}^{(d)}$. We
begin by recalling a classic result on the N\'eron-Severi group of $\mathcal{C}^{(d)}$ (see e.g. \cite[Proposition 5.1, p. 358]{ACGH85}). 

\begin{Prop} \label{prop-nerongeneralmoduli}
	If $\mathcal{C}$ has very general moduli, then the N\'eron-Severi group of $\mathcal{C}^{(d)}$ is a free Abelian group of rank $2$ generated by the numerical classes of the following divisors:
	\begin{itemize}
		\item the divisor given by the image of the inclusion map $i_{d-1}:\mathcal{C}^{(d-1)}\rightarrow\mathcal{C}^{(d)}$, $i_{d-1}(D)=D+P$, where $P\in\mathcal{C}$ is a fixed point;
		\item the pullback of the theta divisor of the Jacobian variety $J\mathcal{C}$ of $\mathcal{C}$ along the Abel-Jacobi map $u:\mathcal{C}^{(d)}\rightarrow J\mathcal{C}$.
	\end{itemize}
\end{Prop}

\begin{Rem}
	In Proposition~\ref{prop-nerongeneralmoduli} and in the rest of this note the statement that the curve $\mathcal{C}$ has \emph{very general moduli} is used as a shorthand for ``there are denumerably many proper subvarieties of the moduli space of genus $g$ curves such that the point corresponding to the curve $\mathcal{C}$ does not belong to anyone of them".
\end{Rem}

Now, we focus on the subgroup $\Pic^0(\mathcal{C}^{d})$.
The following is probably well known to experts, but since we couldn't find a suitable reference in the literature we sketch a proof here.

\begin{Prop}\label{prop-Pic0symmetricproduct}
Let $\mathcal{C}$ be a smooth projective irreducible curve of genus $g$ and let $d\geq2$ be an integer, then  we have an isomorphism of Abelian varieties
\[
\Pic^0(\mathcal{C}^{(d)})\cong \Pic^0(\mathcal{C}) \cong (\mathbb{C}^g/\mathbb{Z}^{2g}).
\]	
\end{Prop}

\begin{proof}
First, we prove that the Jacobian $J\mathcal{C}=\Pic^0(\mathcal{C})$ satisfies the universal property of the Albanese variety for $\mathcal{C}^{(d)}$. Namely, let $f:\mathcal{C}^{(d)}\rightarrow A$ be a morphism to an Abelian variety $A$. We show that $f$ factors through $J\mathcal{C}$.
We compose $f$ with the diagonal embedding $\mathcal{C}\hookrightarrow\mathcal{C}^{(d)}$, obtaining a morphism $f_0:\mathcal{C}\rightarrow A$.
Now fix a base point $P_0\in\mathcal{C}$, and consider the maps $u:\mathcal{C}\rightarrow J\mathcal{C}$ and $\sigma_d:\mathcal{C}^{(d)}\rightarrow J\mathcal{C}$ given by $u(P)=[\mathcal{O}(P-P_0)]$ and $\sigma_d(P_1+\cdots+P_d)=[\mathcal{O}(P_1+\cdots+P_d-dP_0)]$ respectively.
We obtain a commutative diagram
 	\begin{equation*}
 	\begin{tikzcd}
 	& \mathcal{C} \arrow[r, hook] \arrow[rd, "u"] \arrow[rr, bend left=40, "f_0"]
 	&\mathcal{C}^{(d)} \arrow{r}{f}\arrow[d, "\sigma_d"]
 	& A  
 	& \\
 	& 
 	&J\mathcal{C} \arrow[ru, dotted, "h"]
 	& 
 	&  
 	\end{tikzcd}
 	\end{equation*}
 where the morphism $h:J\mathcal{C}\rightarrow A$ exists by the universal property of $J\mathcal{C}\cong\mathrm{Alb}(\mathcal{C})$ and is such that $f_0=h\circ u$ (see e.g. \cite[Section~17.5]{Pol03}).
 One can show that $f$ factors through $h$ as well.
 Therefore we have $\mathrm{Alb}(\mathcal{C}^{(d)})\cong J\mathcal{C}$. 
 Now, recalling that the Albanese and Picard variety are dual complex tori \cite[Chapter~2, Section~6]{GH78} we obtain
 the following chain of isomorphisms
 \[
 \Pic^0(\mathcal{C}^{(d)})\cong\mathrm{Alb}(\mathcal{C}^{(d)})^{\vee}\cong J\mathcal{C}^{\vee}\cong J\mathcal{C}=\Pic^0(\mathcal{C}).
 \]
 Finally, an application of the exponential sequence \cite[Appendix B, Section 5]{Har77} yields the isomorphism $\Pic^0(\mathcal{C}) \cong (\mathbb{C}^g/\mathbb{Z}^{2g})$.	
\end{proof}

\begin{Cor}\label{cor-picardgeneralmoduli}
Let $\mathcal{C}$ be a smooth projective irreducible curve of genus $g$ which has very general moduli and let $d\geq2$ be an integer, then we have
\[
\Pic(\mathcal{C}^{(d)})\cong\Pic^0(\mathcal{C}^{(d)})\oplus \NS(\mathcal{C}^{(d)}) \cong (\mathbb{C}^g/\mathbb{Z}^{2g}) \oplus \mathbb{Z}^2.
\]
\end{Cor}

\begin{proof}
	Consider the short exact sequence of groups
\[
0\rightarrow\Pic^0(\mathcal{C}^{(d)})\rightarrow \Pic(\mathcal{C}^{(d)}) \rightarrow \NS(\mathcal{C}^{(d)}) \rightarrow 0,
\]
where the last map is the first Chern class.
By Proposition~\ref{prop-nerongeneralmoduli}, we have $\NS(\mathcal{C}^{(d)})\cong\mathbb{Z}^2$, therefore the previous sequence splits. Since $\Pic^0(\mathcal{C}^{(d)})\cong \mathbb{C}^g/\mathbb{Z}^{2g}$ by Proposition~\ref{prop-Pic0symmetricproduct}, we are done.
\end{proof}

\begin{Rem}\label{remark-amplenotinPic0}
Let $V$ be a smooth variety of dimension $\geq2$ and let $D$ be an ample divisor on $V$, then  $[D]\not\in\Pic^0(V)$. 
To see this, observe first that if $D_1$ and $D_2$ are two linearly equivalent divisors on $V$ then $D_1\cdot\mathcal{C}\simeq D_2\cdot\mathcal{C}$ for any curve $\mathcal{C}\subseteq V$.
In particular, if $D_1$ is linearly equivalent to $0$, i.e., $[D_1]\in\Pic^0(V)$, then $D_1\cdot\mathcal{C}=0$.
On the other hand, let $n$ be a positive integer such that  $nD$ is very ample and let $V\hookrightarrow\mathbb{P}^r$ be the embedding  through the linear system $|nD|$.
Since the multiples of $nD$ give hyperplanes in $\mathbb{P}^r$, and a generic curve and a hyperplane always meet in projective space, we have $(nD)\cdot\mathcal{C}>0$ generically, but $(nD)\cdot \mathcal{C}=n(D\cdot\mathcal{C})$, so $D\cdot\mathcal{C}>0$, hence $[D]$ cannot be in $\Pic^0(V)$.	
\end{Rem}

We are now ready to provide the claimed construction of a complete normal isolated singularity of dimension $\geq3$ with large class group.
The two-dimensional case was proved by R.~Wiegand in \cite{Wie01}.

\begin{Theorem}\label{theorem-constructionlocalringclassgroup}
	Let $g\geq 1$ and $d\geq 2$ be integers. Then there exists a complete local normal domain $B$ over $\mathbb{C}$ with an isolated singularity such that $\dim B=d+1$ and $\Cl(B)$ contains a copy of $(\mathbb{R}/\mathbb{Z})^{\oplus 2g}$ as subgroup. 	
\end{Theorem}

\begin{proof}
Let $V=\mathcal{C}^{(d)}$ be the symmetric product of a smooth projective irreducible curve of genus $g$ and very general moduli. By Corollary~\ref{cor-picardgeneralmoduli}, we have $\Pic(V)\cong\Pic^0(V)\oplus\NS(V)\cong(\mathbb{C}^g/\mathbb{Z}^{2g})\oplus\mathbb{Z}^2$.
Moreover, since $V$ is a smooth variety we can identify the Picard group of $V$ with its class group, i.e., $\Cl(V)\cong\Pic(V)$.

\par We embed $V$ into a projective space $\mathbb{P}^r=\mathbb{P}^r_{\mathbb{C}}$.
Since $V$ is normal, by eventually composing with a $s$-uple Veronese embedding for $s\gg0$, we may assume that the embedding $V\subseteq\mathbb{P}^r$ is projectively normal, that is the homogeneous coordinate ring $A$ of $V$ is a normal domain of dimension $d+1$ (see \cite[Chapter II, Exercise~5.14]{Har77}).
Moreover since $V$ is smooth, $A$ has only an isolated singularity at the origin. 
	
\par Let $D$ be a very ample divisor on $V$. By Remark \ref{remark-amplenotinPic0}, we know that $[D]\not\in\Pic^0(V)$,
so we have $[D]=(D_1,D_2)\in\Pic^0(V)\oplus\NS(V)$, with $D_2\neq0$.
By the Nakai-Moishezon criterion \cite[Theorem~2.3.18]{Laz04} the ampleness of a divisor is invariant by numerical equivalence, therefore we may assume without loss of generality that $D_1=0$.

\par Let $H$ be the hyperplane of $\mathbb{P}^r$ such that $H\cap V=D$.
By \cite[Chapter II, Exercise~6.3]{Har77} we have a short exact sequence relating the class group of $V$ and the class group of $A$ (which is also the class group of the affine cone $C(V)\subseteq\mathbb{A}^{r+1}$):
\[
0\rightarrow \mathbb{Z} \xrightarrow{\varphi} \Cl(V) \rightarrow \Cl(A) \rightarrow 0.
\]
Here, the first map is given by $\varphi(1)=[H\cap V]=[D]$.
Therefore, we have $\Cl(A)\cong\Cl(V)/\mathrm{Im}\varphi$.
Using the decomposition $\Cl(V)\cong \Pic^0(V)\oplus\NS(V)$, 
we may write $\varphi=(\varphi_1,\varphi_2)$, and we see that $\varphi_1=0$, since $[D]=(0,D_2)$ by assumption.
It follows that $\Cl(A)\cong\Pic^0(V)\oplus\NS(V)/\mathrm{Im}\varphi_2$.
In particular, $\Cl(A)$ contains a copy of $\Pic^0(V)\cong\mathbb{C}^g/\mathbb{Z}^{2g}$ as direct summand.
\par Now, let $\mathfrak{m}$ be the irrelevant maximal ideal of $A$, and consider the localization $A_{\mathfrak{m}}$ and its $\mathfrak{m}A_{\mathfrak{m}}$-adic completion $B=\widehat{A_{\mathfrak{m}}}$.
$B$ is a complete local normal domain of Krull dimension $d+1$ and by the upcoming Lemma~\ref{lemma-completionofisolated}, $B$ is an isolated singularity.
From the localization sequence of the class group we get an isomorphism $\Cl(A)\cong\Cl(A_{\mathfrak{m}})$. Moreover the canonical map $A_{\mathfrak{m}}\rightarrow B$ is faithfully flat, so we obtain an injective group homomorphism $\Cl(A_{\mathfrak{m}})\hookrightarrow \Cl(B)$ which shows that $\Cl(B)$ contains a copy of $\mathbb{C}^g/\mathbb{Z}^{2g}$ as subgroup. Since $\mathbb{C}^g/\mathbb{Z}^{2g}\cong(\mathbb{R}/\mathbb{Z})^{\oplus2g}$ as groups, we are done.
\end{proof}
 
The following lemma is probably well known. I thank T.~Murayama for suggesting it  to me.
 
\begin{Lemma}\label{lemma-completionofisolated}
Let $A$ be a Noetherian algebra of finite type over a field which has only an isolated singularity at a maximal ideal $\mathfrak{m}$. Then the $\mathfrak{m}A_{\mathfrak{m}}$-adic completion of its localization $A_{\mathfrak{m}}$ is an isolated singularity.	
\end{Lemma}

\begin{proof}
	Since $A$ is of finite type over a field it is a $G$-ring in the sense of \cite[p.~255]{Mat89}, and its localization $A_{\mathfrak{m}}$ is a $G$-ring as well, so the map $\varphi:A_{\mathfrak{m}}\rightarrow\widehat{A_{\mathfrak{m}}}$ is regular. Moreover, $A'=A_{\mathfrak{m}}$ is still an isolated singularity. In order to prove that $B=\widehat{A_{\mathfrak{m}}}$ is an isolated singularity, pick a prime ideal $\mathfrak{q}\in\Spec B$ different from the unique maximal ideal, and let $\mathfrak{p}=\mathfrak{q}\cap A'$. 
	The localization $A'_{\mathfrak{p}}$ is a regular local ring, and we are left to prove that $B_{\mathfrak{q}}$ is also regular.
	Consider the localized map 
	\[
	\widetilde{\varphi}:A'_{\mathfrak{p}}\rightarrow B_{\mathfrak{q}}
	\]
	which is still regular. So its fibers are regular rings, and \cite[Theorem~23.7]{Mat89} implies that $B_{\mathfrak{q}}$ is regular.	
\end{proof}

\end{document}